\documentclass[12pt]{article}
\usepackage{graphicx}
\usepackage{amsmath}
\usepackage{amsfonts}
\usepackage{amssymb}
\usepackage{url}
\newtheorem{theorem}{Theorem}[section]

\newtheorem{corollary}[theorem]{Corollary}

\newtheorem{defn}[theorem]{Definition}

\newtheorem{lemma}[theorem]{Lemma}

\newtheorem{proposition}[theorem]{Proposition}
\newtheorem{remark}[theorem]{Remark}

\numberwithin{equation}{section}

\newenvironment{proof}[1][Proof]{\textbf{#1.} }{\ \rule{0.5em}{0.5em}}

\usepackage{color}

\topmargin=-0.5in \everymath{\displaystyle}
\begin{document}
\baselineskip=18pt

\pagenumbering{arabic}

\begin{center}
{\Large {\bf Analysis of a Reaction Diffusion Model for a Reservoir Supported Spread of Infectious Disease}}

\bigskip

W.E. Fitzgibbon and J.J. Morgan

Department of Mathematics

University of Houston\\
Houston, TX 77204, USA

\vspace{0.2in}
\end{center}

\begin{abstract}
Motivated by recent outbreaks of the Ebola Virus, we are concerned with the role that a vector reservoir plays in supporting the spatio-temporal spread of a highly lethal disease through a host population. In our context, the reservoir is a species capable of harboring and sustaining the pathogen. We develop models that describe the horizontal spread of the disease among the host population when the host population is in contact with the reservoir and when it is not in contact with the host population.  These models are of reaction diffusion type, and they are analyzed, and their long term asymptotic behavior is determined.  
\end{abstract}

\medskip


\section{Introduction}
In what follows we introduce a suite of models of increasing complexity to describe the role of a reservoir in circulating a highly lethal disease among a host population. In our context the reservoir will be a species in which the infectious agent normally lives and multiplies. We assume the reservoir harbors the infectious agent without injury to itself and is capable of transmitting the agent to the host species.  We model the outbreak and spatio-temporal spread of a highly lethal disease that is initiated and sustained via contact of a host population with an infected reservoir population.  We will be able to account for the random dispersion of the host population across a region $\Omega$ and the random dispersion of the reservoir species across a proper subregion of $\Omega$.  We will assume that the infectious agent is transferred to the host by virtue of contact with the infected reservoir and that the agent can be transferred horizontally within the host population via contact of infected and uninfected hosts.  Circulation of the infectious agent within the reservoir population occurs horizontally by contact of the uninfected and the infected. We assume no vertical transmission within either population. Our models do not feature standard crisscross dynamics \cite{bib4}.  They do not allow transfer of the infectious agents from the infected host to the uninfected reservoir.   We are motivated by outbreaks of the Ebola Virus Disease (also known as Ebola Hemorrhagic Fever) which is commonly called Ebola. 

The first recorded outbreak of Ebola occurred 1976 and between this and the 2013-2015 West African outbreak, 24 outbreaks have occurred, \cite{bib1}. Ebola affects humans, primates, fruit bats and other mammals. The natural reservoir of the Ebola virus is yet to be determined.  However, among a range of accidental animal hosts fruit bats are believed to be the natural Ebola virus reservoir, \cite{bib2}. The virus exhibits a very high pathogenicity among humans, non–human primates and other mammals. However, it does not appear to have much if any deleterious effect on fruit bats. Ebola symptoms include: fever, sore throat, muscular pain, headaches, vomiting, diarrhea, rash, decreased function of the liver and kidneys, and bleeding (internal and external). The infectious phase starts between two days and three weeks after exposure. The disease has a high mortality, killing between 25 and 90 percent of those infected. Within humans, the virus is spread by direct contact with bodily fluids (blood, urine, saliva, sweat, feces, vomit, breast milk and semen). The bodies of individuals who have died from the disease, as well as those suffering from the disease, can transmit. The virus can also be transmitted by objects of clothing that have been contaminated with the bodily fluids of either individuals with the disease, or the bodies of individuals who have succumbed to the disease, \cite{bib3}.  There is additionally a possibility that the virus persists in the semen of men who have recovered from the disease, and that the disease can be propagated via sexual congress. However, we shall not include this mechanism in our modeling considerations.

We feel it incumbent to emphasize that our models are prototypes far too simplistic to provide an accurate depiction of the spread of Ebola and can only serve as a gateway to the study of more comprehensive and realistic models.  We shall use the terms host and reservoir populations and not refer to any specific host or reservoir species.  Our concluding remarks will return the Ebola virus as a topic of discussion.

\section{Spatially Homogeneous Outbreak Model In A Region With No Reservoir}
The first model describes an outbreak of the virus in a host population that is removed from the reservoir. Here we could consider the arrival of infected travelers in a virus free region with no indigenous reservoir from a region inhabited by an infected reservoir. Our model is a simple system of ordinary differential equations that does not include any spatial considerations. 

We use a variant of the well known SEIR model to describe the propagation of the disease. The state variables are $S$, $E$, $I$ and $R$ together with the variable $C$. The variables $S$, $E$, $I$ and $R$ represent the standard SEIR compartments; the Susceptible Class, $S$, consists of individuals who are capable of contracting the disease; the Exposed Class, $E$, consists of individuals who have contracted the diseases but who have not yet become fully infected and are not capable of transmitting it; the Infected Class, $I$, consists of individuals who are fully infected by the virus and are capable of transmitting it; and the Removed Class, $R$, consists of individuals who have perished, those who have acquired immunity by recovery, or those who have become immune by exposure without fully developing the disease. The Removed Class does not affect the dynamics of the systems and there is no need to consider it analytically.  We include it in our discussion because in a natural way it tracks the impact of the disease. The Contaminated Class, $C$, represents contaminated bodies of individuals who have died from the disease. The progression of Ebola can be described by the following system of ordinary differential equations and initial conditions
\begin{subequations}
\begin{align}
\frac{dS}{dt}&=-\sigma IS-\omega SC\label{rhoa}\\
\frac{dE}{dt}&=\sigma IS+\omega SC-\lambda E\label{rhob}\\
\frac{dI}{dt}&=\lambda_1E-\gamma I\label{rhoc}\\
\frac{dC}{dt}&=\gamma_1I-\mu C\label{rhod}\\
\frac{dR}{dt}&=\lambda_2E+\gamma_2I+\mu C\label{rhoe}\\
S(0)&=S_0,E(0)=E_0,I(0)=I_0,C(0)=C_0,R(0)=R_0\label{rhof}
\end{align}
\end{subequations}
We assume the initial populations $S_0,E_0,I_0>0$ and $C_0,R_0\geq 0$, and the parameters $\alpha,\beta,\lambda_1,\lambda_2,\gamma_1,\gamma_2,\sigma,\mu,\omega>0$ with $\lambda=\lambda_1+\lambda_2$  and $\gamma=\gamma_1+\gamma_2$. The total population is given by $P=S+E+I$. Susceptible individuals can contract the disease from contact with either infected individuals or with the contaminated. The transmission rates are given by $\sigma IS$  and $\omega SC$ respectively entering the exposed class. The virus incubates during the exposed stage and is noninfectious and essentially asymptomatic. Individuals leave the  Exposed Class with a constant rate $\lambda>0$, either becoming fully infected and infectious at rate $\lambda_1>0$, or they are able to subdue the virus and leave the exposed class for the removed class, remaining non-infectious and acquiring immunity at rate  $\lambda_2>0$.   Infectious individuals leave the infectious class at a constant rate $\gamma>0$, either perishing and entering the contaminated class or recovering. Those perishing enter the contaminated class at the rate $\gamma_1>0$, and those recovering gain permanent immunity entering the removed class at rate $\gamma_2>0$.  Contact between susceptibles and the contaminated will expose them to the virus and further deplete the susceptible class.  Contamination decays at a constant rate $\mu>0$.   We formally state the following result and include its proof only in an effort to set the stage for what follows.

\begin{theorem}\label{theorem1}
Assume $\alpha,\beta,\lambda_1,\lambda_2,\gamma_1,\gamma_2,\sigma,\mu,\omega>0$  with $\lambda=\lambda_1+\lambda_2$  and $\gamma=\gamma_1+\gamma_2$. If $S_0,E_0,I_0>0$ and $C_0,R_0\geq 0$, there exists a unique, uniformly bound, nonnegative global solution $S(t),E(t),I(t),C(t),R(t)$ to (\ref{rhoa}-\ref{rhof}) with $$\max_{t\ge 0}\{S(t),E(t),I(t),C(t)\}\leq S_0+E_0+I_0+C_0$$In addition, $$\lim_{t \to \infty}E(t)=\lim_{t\to\infty}I(t)=\lim_{t\to\infty}C(t)=0$$and$$\lim_{t\to\infty}S(t)=S_*>0$$
\end{theorem}

\begin{proof}
The analysis of this system relies on elementary methods of ordinary differential equations.  A variety of different fixed point theorems guarantee the existence of unique solutions on a maximal interval $[0,T_\text{max})$.  Showing that the solution components are uniformly bounded on any bounded interval $[0,T_\text{max})$ will insure that $T_\text{max}=\infty$.  We observe that the vector field $$F(S,E,I,C)=\begin{bmatrix}-\sigma IS-\omega SC\\\sigma IS+\omega SC-\lambda E\\\lambda_1E-\gamma I\\\gamma_1I-\mu C\\\lambda_2E+\gamma_2I+\mu C\end{bmatrix}$$does not point out of the positive orthant of $\mathcal{R}^5$. Consequently, since all the initial data is componentwise nonnegative, the solutions to (\ref{rhoa}-\ref{rhof}) are nonnegative for $t\in[0,T_\text{max})$. If we define $$U(t)=S(T)+E(t)+I(t)+C(t)+R(t)$$ then we can sum the equations in (\ref{rhoa}-\ref{rhoe}) to observe $\frac{dU}{dt}\leq 0$, and consequently, $U(t)\leq U(0)$. As a result, there exists an $M>0$ so that $$S(t),E(t),I(t),C(t),R(t)\leq S_0+E_0+I_0+C_0+R_0 \text{ for all }t\in[0,T_\text{max})$$Hence, $T_\text{max}=\infty$, and we are assured unique, uniformly bound, nonnegative global solutions. Moreover, there exists $N>0$ so that
$$\max \left\{\left|\frac{dS}{dt}\right|,\left|\frac{dE}{dt}\right|,\left|\frac{dI}{dt}\right|,\left|\frac{dS}{dC}\right|\right\} < N$$
We set $W=S+E+I+C$ and observe that
$$\frac{dW}{dt}+\lambda_2 E+\gamma_2 I+\mu C=0$$
and hence $\int_0^\infty\left(E(s)+I(s)+C(s)\right)ds < \infty$. Coupling this with the nonnegativity of $E$, $I$ and $C$, and the boundedness of their derivatives, implies 
$$\lim_{t \to \infty}E(t)=\lim_{t\to\infty}I(t)=\lim_{t\to\infty}C(t)=0$$
Now observe that if $g(s)=\sigma I(S)+\omega C(S)$, then $\int_0^\infty g(s)ds = K < \infty$. Also, $\frac{dS}{dt}=-g(t)S$. Consequently,
$$S(t)=S_0e^{-\int_0^tg(s)ds}\ge S_0 e^{-K}>0$$
Therefore, since $S(T)$ is nonincreasing, and bounded below by $S_0e^{-K}>0$, we area assumed that 
$$\lim_{t\to\infty}S(t)=S_*>0$$
\end{proof}

\section{Reservoir Supported Spread}
In this section we consider the case of an infected reservoir introducing and supporting the spread of a virus through a host population. In this case, susceptible hosts can contract the disease via contact with an infected reservoir, as well as horizontally from infected and contaminated hosts. Interspecies transmission occurs only from infected reservoir to susceptible hosts.  We do not assume that infected or contaminated hosts transmit the virus back to the reservoir. 

Our model accounts for disease dynamics and demographics within the reservoir. We assume that the reservoir demographics are modeled by the standard logistic equation. We let $\theta(t)$ denote the time dependent reservoir population having birth rate $\beta$ and logistic mortality coefficient $m$.  Both $\beta$ and $m$ are positive constants. The reservoir population $\theta(t)$ satisfies 
\begin{subequations}
\begin{align}
\frac{d\theta}{dt}&=\beta\theta-m\theta^2\label{thetaa}\\
\theta(0)&=\theta_0>0\label{thetab}
\end{align}
\end{subequations}
The presence of the virus subdivides the reservoir population two compartments, the uninfected, $\phi(t)$, and the infected $\psi(t)$, where $\theta(t)=\phi(t)+\psi(t)$.  It is of course known that $\theta(t)>0$ for $t>0$, and moreover we have $\lim_{t\to\infty} \theta(t)=\frac{\beta}{m}$. If $\sigma_1>0$, we introduce a force of infection term $f(\sigma,\psi)=\sigma_1 \phi \psi$ to model the transmission of the virus from the infected to the uninfected in the reservoir.  We have:
\begin{subequations}
\begin{align}
\frac{d\phi}{dt}&=\beta \theta-\sigma_1\phi\psi\label{phia}\\
\frac{d\psi}{dt}&=\sigma_1\phi\psi-m\psi\theta\label{phib}\\
\phi(0)&=\phi_0>0;\psi(0)=\psi_0>0\label{phic}
\end{align}
\end{subequations}
This system assumes that there is no vertical transmission to reservoir offspring, and that the virus has no deleterious effect on the reservoir. An invariant rectangles argument shows that the solution components remain non-negative.  In addition, since $0\le\theta(t)\le\max\left\{\frac{\beta}{m},\theta_0\right\}$ for all $t \ge 0$, we can use the fact that $\theta(t)=\phi(t)+\psi(t)$ to conclude that  $0\le\phi(t),\psi(t)\le\max\left\{\frac{\beta}{m},\theta_0\right\}$ for all $t \ge 0$.

We will again use the state variables $\left\{S(t),E(t),I(t),C(t),R(t)\right\}$ for the host population, and the differential equations describing the circulation of the disease within the host population are basically the same, differing only by the inclusion of interspecies transmission from the vector reservoir. We introduce a term of the form $\kappa S\psi$ ($\kappa >0$) to account for transmission from the infected reservoir. This becomes a loss term for the susceptible hosts, and gain term for the exposed hosts.

Again a variant of the SEIR model is used to describe the propagation of the disease within the host population. We have variables $S$, $E$, $I$ and $R$ together with the variable $C$. $S$, $E$, $I$ and $R$ represent the standard SEIR components; the susceptible class, the exposed class, the infected class (individuals who are fully infected by the virus and capable of transmitting) and the removed class. The Contaminated Class is given by $C$. Susceptible hosts can contract the disease by contact with the vector reservoir, or they can contract it by contact with infected or contaminated individuals.

The force of infection term, which accounts for transmission of the virus to the Susceptible Class from the infected reservoir, as well as transmission from the infective and contaminated hosts, has the form $f_2(S,I,C,\psi)=\sigma_2SI+\kappa S\psi+\omega SC$ where $\sigma_2$, $\kappa$ and $\omega$ are positive constants. The equations that describe the circulation of the pathogen among the host population now become
\begin{subequations}
\begin{align}
\frac{dS}{dt}&=-\sigma_2 IS-\kappa S\psi-\omega SC\label{Sa}\\
\frac{dE}{dt}&=\sigma_2 IS+\kappa S\psi+\omega SC-\lambda E\label{Sb}\\
\frac{dI}{dt}&=\lambda_1E-\gamma I\label{Sc}\\
\frac{dC}{dt}&=\gamma_1I-\mu C\label{Sd}\\
\frac{dR}{dt}&=\lambda_2E+\gamma_2I\label{Se}\\
S(0)&=S_0,E(0)=E_0,I(0)=I_0,C(0)=C_0,R(0)=R_0\label{Sf}
\end{align}
\end{subequations}
Neither the transmission of the virus nor the demographics in the reservoir are affected by the dynamics of the virus in the host.  We are able to analyze the dynamical behavior of the reservoir (\ref{Sa}-\ref{Sf}) independently.

\begin{proposition}\label{prop2-1} If $\left\{\sigma_2,\beta,m,\phi_0,\psi_0\right\}>0$, then there exists a globally bounded non-negative solution pair $\left(\phi(t),\psi(t)\right)$ to (\ref{Sa}-\ref{Sc}), with 
$$\max_{t\ge0}\left\{\phi(t),\psi(t)\right\}\le\max\left\{\frac{\beta}{m},\theta_0=\phi_0+\psi_0\right\}$$
Moreover, if $\sigma_1\le m$ then 
$$\lim_{t \to \infty}\left(\phi(t),\psi(t)\right)=\left(\frac{\beta}{m},0\right)$$
If $\sigma_1> m$ then 
$$\lim_{t \to \infty}\left(\phi(t),\psi(t)\right)=\left(\frac{\beta}{\sigma_1},\frac{\beta}{\sigma_1}\left(\frac{\sigma_1}{m}-1\right)\right)$$
\end{proposition}

\begin{proof} Following the reasoning of Theorem \ref{theorem1}, the structure of the reaction vector field and the non-negativity of $\theta$ guarantees that $\phi(t),\psi(t) \ge 0$. Furthermore, since $\theta(t)$ is uniformly bounded and $\theta(t)=\phi(t)+\psi(t)$, we are assured a globally bounded non-negative solution $\left(\theta(t),\phi(t),\psi(t)\right)$ to (\ref{phia}-\ref{phic}), and from our previous comments, $0\le\phi(t),\psi(t)\le\max\left\{\frac{\beta}{m},\theta_0\right\}$ for all $t\ge0$. In addition, from our earlier discussion, we have $\lim_{t\to\infty}\theta(t)=\frac{\beta}{m}$. Consequently, we can determine the asymptotic behavior of $\psi(t)$ by substituting $\phi(t)=\theta(t)-\psi(t)$ and replacing $\theta(t)$ with $\theta(t)-\frac{\beta}{m}+\frac{\beta}{m}$ in (\ref{phib}). Rewriting the equation gives
\begin{equation}
\frac{d\psi}{dt}=\left(\sigma_1-m\psi\right)\left(\theta(t)-\frac{\beta}{m}\right)+\beta\left(\frac{\sigma_1}{m}-1\right)\psi-\sigma_1\psi^2\label{psi-eq}
\end{equation}
If we define $\epsilon(t)=\left(\sigma_1-m\psi\right)\left(\theta(t)-\frac{\beta}{m}\right)$, then $\lim_{t\to\infty}\epsilon(t)=0$. 
If $\sigma_1\le m$ then $\frac{d\psi}{dt}\le \epsilon(t)-\sigma_1\psi^2$. As a result, the nonnegativity of $\psi(t)$ and $\psi(0)>0$ imply that $\lim_{t\to\infty}\psi(t)=0$. Consequently, $\lim_{t\to\infty}\phi(t)=\lim_{t\to\infty}\theta(t)=\frac{\beta}{m}$. If $\sigma_1>m$, then the nonnegativity of $\psi(t)$, (\ref{psi-eq}) and $\psi(0)>0$ imply $\lim_{t\to\infty}\psi(t)=\frac{\beta}{\sigma_1}\left(\frac{\sigma_1}{m}-1\right)$. Consequently, since $\lim_{t\to\infty}\theta(t)=\frac{\beta}{m}$ and $\theta(t)=\phi(t)+\psi(t)$, we obtain $\lim_{t\to\infty}\phi(t)=\frac{\beta}{\sigma_1}$.
\end{proof}

\begin{theorem}\label{theorem2} Assume $\alpha,\beta,\lambda_1,\lambda_2,\gamma_1,\gamma_2,\sigma,\mu,\omega,\kappa >0$ with $\lambda=\lambda_1+\lambda_2$ and $\gamma=\gamma_1+\gamma_2$. If $S_0,E_0,I_0,\phi_0,\psi_0>0$ and $C_0,R_0\ge0$, then there exists a globally bounded non-negative solution $\left\{S(t),E(t),I(t),C(t),\phi(t),\psi(t)\right\}$ to (\ref{phia}-\ref{phic},\ref{Sa}-\ref{Sf}), with 
$$\max_{t\ge 0}\left\{S(t),E(t),I(t),C(t)\right\}\le\left\{S_0+E_0+I_0+C_0\right\}$$
and
$$\max_{t\ge 0}\left\{\phi(t),\psi(t)\right\}\le\max\left\{\phi_0+\psi_0,\frac{\beta}{m}\right\}$$
Moreover, $$\lim_{t \to \infty}E(t)=\lim_{t\to\infty}I(t)=\lim_{t\to\infty}C(t)=0$$
If $\sigma_1\le m$ then $\lim_{t\to\infty}S(t)=S_*>0$ and
$$\lim_{t \to \infty}\left(\phi(t),\psi(t)\right)=\left(\frac{\beta}{m},0\right)$$
If $\sigma_1> m$ then $\lim_{t\to\infty}S(t)=0$ and 
$$\lim_{t \to \infty}\left(\phi(t),\psi(t)\right)=\left(\frac{\beta}{\sigma_1},\frac{\beta}{\sigma_1}\left(\frac{\sigma_1}{m}-1\right)\right)$$
\end{theorem}
\begin{proof}The arguments establishing global existence, non-negativity, and uniform bounds, as well as the convergence of $E(t)$, $I(t)$ and $C(t)$, are essentially the same as those given in Theorem \ref{theorem1} and will not be repeated. The dynamics of $\left\{\phi(t),\psi(t)\right\}$ is independent of the host population and its asymptotic behavior is given by Proposition \ref{prop2-1}. All that remains is the asymptotic behavior of $S(t)$.  However, when $\sigma_1\le m$, we know that $\lim_{t\to\infty}\psi(t)=0$, and as a result, the asymptotic behavior is determined by (\ref{rhoa}-\ref{rhoe}), implying $\lim_{t\to\infty}S(t)=S_*>0$. In the case of $\sigma_1>m$, we know $\lim_{t\to\infty}\psi(t)=\frac{\beta}{\sigma_1}\left(\frac{\sigma_1}{m}-1\right)>0$, and consequently, (\ref{Sa}) implies $S^\prime(t)\le-k_1S(t)$ for some $k_1>0$. Therefore, $\lim_{t\to\infty}S(t)=0$.
\end{proof}

\section{Spatial Outbreak Model}
The next model is concerned with a dispersing spatially distributed population.  We assume that our population is confined to a bounded region $\Omega$ in $R^2$ with smooth boundary $\partial\Omega$, such that $\Omega$ lies locally on one side of $\partial\Omega$.  The state variables $\left\{s,e,i,c\right\}$ represent time dependent spatial densities of the susceptible, exposed, infective, and contaminated classes.  The time dependent populations of these classes are obtained by integration over $\Omega$. That is
$$S(t)=\int_\Omega s(x,t)dx, E(t)=\int_\Omega e(x,t)dx, I(t)=\int_\Omega i(x,t)dx, C(t)=\int_\Omega c(x,t)dx$$
Population dispersion is modeled by Fickian diffusion with spatially dependent diffusivities $d_2$ and $d_3$ for the susceptible and exposed classes respectively. We further assume that $d_2,d_3>0$ and are smooth on $\overline\Omega$. We remark that for the sake of mathematical completeness we allow distinct diffusivities for suspectible and exposed densities. Since we assume that exposed individuals exhibit no adverse effects from the disease, this is unnecessary for our model. However, since infective/infectious are severely ill, we assume spatial distribution but no diffusion. Similarly, contaminated individuals are spatially dispersed with no diffusion. The boundary conditions, $\frac{\partial s}{\partial \eta}=\frac{\partial e}{\partial \eta}=0$ on $\partial\Omega$ for $t>0$, insure that the populations remain confined to $\Omega$ for all time. The infection kinetics are essentially the same as those given for the initial ordinary differential equation, with the only difference being that we assume the contact rates $\sigma(x)s(x,t)i(x,t)$ and $\omega(x) s(x,t)c(x,t)$ are spatially dependent. These considerations give rise to the following coupled system of partial differential equations and evolving ordinary differential equations:
\begin{subequations}
\begin{align}
\frac{\partial s}{\partial t}&=\bigtriangledown\cdot d_2(x)\bigtriangledown s-\sigma(x) si-\omega(x)sc &x \in\Omega,t>0 \label{odepde1a}\\
\frac{\partial e}{\partial t}&=\bigtriangledown\cdot d_3(x)\bigtriangledown e+\sigma(x) si+\omega(x) sc-\lambda e &x \in\Omega,t>0\label{odepde1b}\\
\frac{\partial i}{\partial t}&=\lambda_1e-\gamma i &x \in\Omega,t>0\label{odepde1c}\\
\frac{\partial c}{\partial t}&=\gamma_1i-\mu c &x \in\Omega,t>0\label{odepde1d}\\
\frac{\partial s}{\partial \eta}&=\frac{\partial e}{\partial \eta}=0 &x \in\partial\Omega,t>0\label{odepde1e}\\
s(x,0)&=s_0(x),e(x,0)=e_0(x),i(x,0)=i_0(x),c(x,0)=c_0(x) &x \in\Omega\label{odepde1f}
\end{align}
\end{subequations}
Recall that $\lambda e$ is the rate of transfer from the exposed class, $\lambda=\lambda_1+\lambda_2$, and the rate of transfer from the exposed class to the recovered class by warding off the infection is given by $\lambda_2 e$. The infection induced mortality is $\gamma i$, $\gamma=\gamma_1+\gamma_2$, and $\gamma_2 i$ represents recovery. The time dependent removed population $R(t)$, which includes those who have become immune through exposure, and those who have become immune through recovery as well as those who have perished satisfies,
$$\frac{dR}{dt}=\lambda_2\int_\Omega e(x,t)dx+\gamma_2\int_\Omega i(x,t)dx+\mu\int_\Omega c(x,t)dx$$
$R(t)$ gives a measure of the impact of the infection, but does not feed back into the dynamics of the system, and therefore will not be given further discussion. These considerations produce the system (\ref{odepde1a}-\ref{odepde1f}) of coupled semi linear diffusion equations with parametrically evolving ordinary differential equations. A system such as (\ref{odepde1a}-\ref{odepde1f}) is often called partially dissipative \cite{bib5,bib6}. 

We shall make use of the following lemma which guarantees uniform a priori $C^{2,1}(\overline\Omega\times(0,\infty))$ bounds for solutions to the nonhomogeneous linear diffusion equations.  We extract the details of the proof from arguments for more general results appearing in \cite{bib8} which are predicated on fundamental parabolic theory of \cite{bib9}. 

\begin{lemma}\label{lemma3-2} Assume that $\Omega$ is a bounded region in $R^2$ with smooth boundary $\partial\Omega$ such that $\Omega$ lies locally on one side of $\partial\Omega$, $d\in C^2(\overline\Omega,(0,\infty))$ and $v_0\in C^2(\overline\Omega)$. If $p>1$, $T>0$ and $f\in L_p(\Omega\times(0,T))$, then there exists a unique $v\in W_p^{(2,1)}(\Omega\times(0,T))$ solving
\begin{align}
\frac{\partial v}{\partial t}&=\bigtriangledown\cdot d(x)\bigtriangledown v+f(x,t)&x\in\Omega,t>0\nonumber\\
\frac{\partial v}{\partial \eta}&=0&x\in\partial\Omega,t>0\nonumber\\
v(x,0)&=v_0(x)&x\in\Omega\nonumber
\end{align}
and there exists a constant $C_{p,T}>0$, independent of $f$ and $v_0$ so that 
\begin{equation}\label{3-3}
\|v\|_{p,\Omega\times(0,T)}^{(2,1)}\le C_{p,T}\left(\|v_0\|_\infty^{(2)}+\|f\|_{p,\Omega\times(0,T)}\right)
\end{equation}
Furthermore, if $p>\frac{n+2}{2}$ and $f\in L_p\left(\Omega\times(\tau,\tau+2)\right)$ for all $\tau>0$, then $v\in C(\overline\Omega\times[\tau+1,\tau+2])$ and there exists a constant $N_p$ independent of $f$ and $v_0$ so that 
\begin{equation}\label{3-4}
\|v\|_{\infty,\Omega\times(\tau+1,\tau+2)}\le N_{p}\left(\|f\|_{p,\Omega\times(\tau,\tau+2)}+\|v\|_{p,\Omega\times(\tau,\tau+2)}\right)
\end{equation}
Finally, if $f\in C^{\left(\theta,\frac{\theta}{2}\right)}(\overline\Omega\times[0,T])$ for some $0<\theta<1$, then $v\in C^{(2,1)}(\overline\Omega\times(0,T])$, and if $\epsilon>0$ then there exists $M_{\theta,\epsilon,T}>0$ independent of $f$ and $v_0$ so that 
\begin{equation}\label{3-5}
\|v\|_{\infty,\Omega\times(\epsilon,T)}^{(2,1)}\le M_{\theta,\epsilon,T}\left(\|v_0\|_\infty^{(2)}+\|f\|_{\infty,\Omega\times(0,T)}^{\left(\theta,\frac{\theta}{2}\right)}\right)
\end{equation}
\end{lemma}
\begin{proof}
(\ref{3-3}) follows from \cite{bib9} pg 341. Now, suppose $p>\frac{n+2}{2}$, $\tau\ge 0$ and $f\in L_p(\Omega\times(\tau,\tau+2))$. Let $h\in C^\infty([\tau,\tau+2],[0,1])$ such that $h(0)=0$ and $h(t)=1$ for all $\tau+1\le t\le\tau+2$. Define $w(x,t)=v(x,t)h(t)$. Then $w$ solves 
\begin{align}
\frac{\partial w}{\partial t}&=\bigtriangledown\cdot d(x)\bigtriangledown w+hf+vh^\prime&x\in\Omega,\tau<t<\tau+2\nonumber\\
\frac{\partial w}{\partial \eta}&=0&x\in\partial\Omega,\tau<t<\tau+2\nonumber\\
w(x,0)&=0&x\in\Omega\nonumber
\end{align}
Then estimate (\ref{3-3}) applied to this system implies
\begin{align}
\|w\|_{p,\Omega\times(\tau,\tau+2)}^{(2,1)}&\le C_{p,2}\|hf+vh^\prime\|_{p,\Omega\times(\tau,\tau+2)}\|\nonumber\\
&\le K C_{p,2}\left( \|f\|_{p,\Omega\times(\tau,\tau+2)}+\|v\|_{p,\Omega\times(\tau,\tau+2)}\right)\nonumber
\end{align}
where $K=\|h^\prime\|_{\infty,(\tau,\tau+2)}$. Since $p>\frac{n+2}{2}$, it follows from \cite{bib9} that $W_p^{(2,1)}(\tau+2,\tau+2))$ imbeds continusly into $L_\infty(\Omega\times(\tau+2,\tau+2))$, implying (\ref{3-4}) holds. (\ref{3-5}) follows from a similar argument by employing classical estimates and Sobolev imbedding.
\end{proof}

The following requirements will guarantee the global well-posedness and uniform boundedness of solutions to (\ref{odepde1a}-\ref{odepde1f}).

\medskip
\noindent\textbf{Condition I:}
\begin{itemize}
\item $\Omega$ is a bounded region in $R^2$ with smooth boundary $\partial\Omega$ such that $\Omega$ lies locally on one side of $\partial\Omega$.
\item $d_2,d_3\in C^2(\overline\Omega)$ and $d_2,d_3\ge\delta$ for some $\delta>0$.
\item $\lambda,\lambda_1,\lambda_2>0$ with $\lambda=\lambda_1+\lambda_2$; $\gamma,\gamma_1,\gamma_2>0$ with $\gamma=\gamma_1+\gamma_2$; and $\mu>0$. 
\item $\omega,\sigma\in C^1(\overline\Omega)$ with $\omega,\sigma>\alpha$ for some $\alpha>0$.
\end{itemize}
We can insure that (\ref{odepde1a}-\ref{odepde1f}) is well posed with globally bounded classical solutions.
\begin{theorem}\label{theorem3}
Assume that all requirement of Condition I are satisfied. If $s_0,e_0,i_0,c_0\in C^2(\overline\Omega)$ are nonnegative and $s_0$ is not identically $0$, then there exists a unique, componentwise nonnegative, uniformly sup norm bounded classical solution $s,e,i,c\in C^{(2,1)}(\overline\Omega\times(0,\infty))$ to (\ref{odepde1a}-\ref{odepde1f}). Furthermore, there exists a constant $s_*>0$ such that 
$$\lim_{t\to\infty}\|s(\cdot,t)-s_*\|_{\infty,\Omega}=\lim_{t\to\infty}\|e(\cdot,t)\|_{\infty,\Omega}=\lim_{t\to\infty}\|i(\cdot,t)\|_{\infty,\Omega}=\lim_{t\to\infty}\|c(\cdot,t)\|_{\infty,\Omega}=0$$
\end{theorem}
\begin{proof}
The weakly coupled system is structurally similar to a Hodgkin Huxley system, and we can adapt the Green’s function/variation of parameters methods cf. \cite{bib12,bib13} to establish the local well posedness, and a maximal time interval of existence $[0,T_{max})$. Moreover, solutions to (\ref{odepde1a}-\ref{odepde1f}) exist globally (i.e. $T_{max}=\infty$), provided they do not blow up in the sup norm in finite time. Because the crux of the global existence argument relies on establishing a priori bounds, we shall not provide the specific details of the argument insuring that $T_{max}=\infty$. Instead we will assume that our solutions are globally well posed and focus on the argument producing these bounds, and given this assumption, produce uniform a prior bounds. The invariant region theory of \cite{bib14} insures that the solutions remain nonnegative. In addition, it is immediately clear that the maximum principle implies $\|s(\cdot,t)\|_{\infty,\Omega}\le\|s_0\|_{\infty,\Omega}$. Also, if we sum the equations above and integrate over $\Omega$, we have 
$$\frac{d}{dt}\int_\Omega (s+e+i+c)dx+\min\left\{\mu,\lambda_2,\gamma_2\right\}\int_\Omega(e+i+c)dx\le0$$
As a result, we have a uniform $L_1(\Omega)$ bound for each component of our system, and 
\begin{equation}\label{3-6}
\lim_{t\to\infty}\int_T^{T+1}\int_\Omega\left(e(x,t)+i(x,t)+c(x,t)\right)dxdt=0
\end{equation}
In addition, the sum of the components of the reaction vector field associated with $s$ and $e$ is nonpositive, and $s$ is bounded. So, regardless of the presence of $i$ and $c$, the uniform boundedness results of \cite{bib8} can be adapted to guarantee that for every $1<p<\infty$, there is a constant $C_p>0$ providing a uniform space time cylinder bound
\begin{equation}\label{3-7}
\|e\|_{p,\Omega\times(T,T+1)}\le C_p\text{ for all }T\ge0
\end{equation}
Consequently, if $q>1$ then
$$\int_T^{T+1}\int_\Omega e(x,t)^qdxdt\le\sqrt{\int_T^{T+1}\int_\Omega e(x,t)dxdt}\sqrt{\int_T^{T+1}\int_\Omega e(x,t)^{2q-1}dxdt)}
$$
Therefore, if we apply (\ref{3-6}), (\ref{3-7}) and the nonnegativity of $e$, we can conclude that 
\begin{equation}\label{3-8}
\|e\|_{q,\Omega\times(T,T+1)}\to 0\text{ as }T\to\infty\text{ for all }q\ge 1
\end{equation}
If we multiply (\ref{odepde1c}) by $i^{p-1}$ and invoke Young's inequality, we can observe that if $\epsilon>0$ then there exists $K_\epsilon>0$ so that
\begin{equation}\label{3-9}
\frac{1}{p}\frac{\partial}{\partial t}i^p=\sigma ei^{p-1}-\gamma i^p\le K_\epsilon \sigma e^p-(\gamma-\epsilon \sigma)i^p
\end{equation}
So, if we choose $\epsilon$ so that $0<\epsilon\sigma<\frac{\gamma}{2}$, we have
\begin{equation}\label{3-10}
\frac{\partial}{\partial t}i^p\le pK_\epsilon \sigma e^p-\frac{p\gamma}{2}i^p
\end{equation}
Therefore, if $a=\frac{p\gamma}{2}$ and $b=pK_\epsilon\sigma$, then
\begin{equation}\label{3-11}
i(x,t)^p\le \exp(-at)i_0(x)^p+b\int_0^t \exp(-a(t-\tau))e(x,\tau)^p d\tau
\end{equation}
Let $k$ be a nonnegative integer. Then
\begin{align}\label{3-12}
\int_k^{k+1}\int_\Omega i(x,t)^pdxdt&\le \int_k^{k+1}\int_\Omega exp(-at)i_0(x)^pdxdt\\
&+b\int_k^{k+1}\int_\Omega\int_0^t\exp(-a(t-\tau))e(x,\tau)^p dxdt\nonumber
\end{align}
We examine the second term on the right hand side of (\ref{3-12}) and obtain
\begin{align}
\int_k^{k+1}\int_\Omega\int_0^t\exp(-a(t-\tau))&e(x,\tau)^p d\tau dxdt=\int_k^{k+1}\sum_{j=1}^k\int_\Omega\int_{j-1}^j\exp(-a(t-\tau))e(x,\tau)^p d\tau dxdt\nonumber\\
&+\int_k^{k+1}\int_\Omega\int_k^t\exp(-a(t-\tau))e(x,\tau)^p d\tau dxdt\nonumber\\
&\le\sum_{j=1}^k\int_k^{k+1}\int_\Omega\int_{j-1}^j\exp(-a(j-1))e(x,\tau)^pd\tau dxdt+(C_p)^p\nonumber\\
&\le(C_p)^p\left(\sum_{j=1}^k\exp(-a(j-1))+1\right)\nonumber\\
&\le(C_p)^p\left(\frac{1}{1-\exp(-a)}+1\right)\nonumber
\end{align}
As a result, for any $p>1$ there exists $C_{e,p}>0$ such that 
\begin{equation}\label{3-13}
\|e\|_{p,\Omega\times(k,k+1)}\le C_{e,p}
\end{equation}
independent of $k$. Consequently, similar to (\ref{3-8}), we can conclude that 
\begin{equation}\label{3-14}
\|e\|_{q,\Omega\times(T,T+1)}\to 0\text{ as }T\to\infty\text{ for all }q\ge1
\end{equation}
In a similar fashion, we can conclude that for any $p>1$ there exists $C_i,p>0$ so that 
\begin{equation}\label{3-15}
\|i\|_{p,\Omega\times(k,k+1)}\le C_{i,p}
\end{equation}
independent of $k$. Consequently, similar to (\ref{3-8}) and (\ref{3-14}), we can conclude that 
\begin{equation}\label{3-16}
\|i\|_{q,\Omega\times(T,T+1)}\to 0\text{ as }T\to\infty\text{ for all }q\ge1
\end{equation}
Now we return to equation (\ref{odepde1b}) for $e$, and view the nonlinearity as a forcing term
$$f(x,t)=\sigma s(x,t)i(x,t)+\omega s(x,t)c(x,t)-\lambda e(x,t)$$
Then Lemma \ref{lemma3-2}, (\ref{3-4}) and the results above imply $\|e\|_{\infty,\Omega}$ is bounded independent of $t$, and 
\begin{equation}\label{3-17}
\|e\|_{\infty,\Omega}\to 0\text{ as }t\to\infty
\end{equation}
We can apply this estimate to (\ref{odepde1c}) to conclude that $\|i(\cdot,t)\|_{\infty,\Omega}$ is bounded independent of $t$, and 
\begin{equation}\label{3-18}
\|i\|_{\infty,\Omega}\to 0\text{ as }t\to\infty
\end{equation}
Finally, we can apply this estimate to (\ref{odepde1d}) to conclude that $\|c(\cdot,t)\|_{\infty,\Omega}$ is bounded independent of $t$, and 
\begin{equation}\label{3-19}
\|c\|_{\infty,\Omega}\to 0\text{ as }t\to\infty
\end{equation}
Now, we reconsider (\ref{odepde1a}). Multiplying both sides by $s$ and integrating by parts gives
$$\int_\Omega s(x,T)^2dx+2\int_0^T\int_\Omega d_2(x)|\bigtriangledown s(x,t)|^2dxdt\le\int_\Omega s_0(x)^2dx$$
implying 
\begin{equation}\label{3-20}
\int_T^{T+1}\int_\Omega d_2(x)|\bigtriangledown s(x,t)|^2 dxdt\to 0\text{ as }T\to\infty
\end{equation}
If we now multiply both sides of (\ref{odepde1a}) by $s_t$, integrate by parts, employ the uniform bound on $s$, (\ref{3-17}), (\ref{3-18}), (\ref{3-19}) and (\ref{3-20}), we can conclude that there exists $K>0$ so that 
\begin{equation}\label{3-21}
\frac{d}{dt}\int_\Omega|\bigtriangledown s(x,t)|^2 dx\le K\text{ for all }t>0
\end{equation}
Finally, from (\ref{odepde1a}), we can see that $\frac{d}{dt}\int_\Omega s(x,t)dx\le 0$. As a result, since $s\ge0$, we can conclude that there exists $s_*\ge0$ so that $\frac{1}{\Omega}\int_\Omega s(x,t)dx\to s_*$ as $t\to\infty$. As a result, (\ref{3-20}) implies $\|s(\cdot,t)-s_*\|_{2,\Omega}\to0$ as $t\to\infty$. Then the standard boot strapping can be employed to guarantee that $\|s(\cdot,t)-s_*\|_{\infty,\Omega}\to0$ as $t\to\infty$. We now claim that $s_*>0$. Since $s_0$ is not identically $0$, standard arguments imply $s(x,t)>0$ for all $x\in\Omega$ and $t>0$, and there exist $\epsilon,t_0>0$ such that $s(x,t_0)\ge\epsilon$ for all $x\in\Omega$. We define $w(x,t)=s(x,t)-\ln(s(x,t))$ for $x\in\Omega$ and $t\ge t_0$, and note that $w(x,t)\ge0$ for $x\in\Omega$ and $t\ge t_0$. Moreover, we may observe that
\begin{subequations}
\begin{align}
w_t&=\left(1-\frac{1}{s}\right)s_t\label{3-22a}\\
\bigtriangledown\cdot d_2(x)\bigtriangledown w&=\left(1-\frac{1}{s}\right)\bigtriangledown\cdot d_2(x)\bigtriangledown s+\frac{1}{s^2}d_2(x)|\bigtriangledown s|^2\label{3-22b}
\end{align}
\end{subequations}
It is also clear that 
\begin{equation}\label{3-22c}
\frac{\partial}{\partial \eta}w(x,t)=0\text{ for all }x\in\partial\Omega,t\ge t_0
\end{equation}
Therfore, $w(x,t)$ satisfies
\begin{align}
w_t&\le\bigtriangledown\cdot d_2(x)\bigtriangledown w+f(x,t)&x\in\Omega,t\ge t_0\nonumber\\
\frac{\partial}{\partial \eta}w&=0&x\in\partial\Omega,t\ge t_0\nonumber\\
w(x,t_0)&\le\|s(\cdot,t_0)\|_{\infty,\Omega}-\ln(\epsilon)&x\in\Omega\nonumber
\end{align}
where $f=-(\sigma si+\omega sc)+(\sigma i+\omega c)$. From the estimates above, we know
$$\int_0^\infty\int_\Omega f(x,t)dxdt<\infty$$
As a result, we can use the comparison principle to observe that 
$$\|w(\cdot,t)\|_{1,\Omega}\le|\Omega|\left(\|s(\cdot,t_0)\|_{\infty,\Omega}-\ln(\epsilon)\right)+\int_{t_0}^\infty\int_\Omega f(x,t)dxdt<\infty$$
for all $t\ge t_0$. Therefore, from the definition of $w$ and the asymptotic behavior of $s$, we know that $s_*-\ln(s_*)<\infty$. Therefore, $s_*>0$. 
\end{proof}
\begin{remark}
The assumption in Theorem \ref{theorem3} that the initial data is smooth can be relaxed to allow the initial data to only be nonnegative and continuous on $\overline\Omega$, provided we are not interested in classical solutions. In this case, we can still obtain the global well posedness, uniform bounds and asymptotic behavior in Theorem \ref{theorem3}, but the $i$ and $c$ components will have no more smoothness in the spatial variable than their initial data, since any lack of smoothing present in the initial data for $i$ and $c$ will propagate with time.  In the absence of initial smoothness of $i$ and $c$, the argument for the persistence of $s(x,t)$ becomes more involved. The $s$ and $e$ components will immediately become smooth regardless, due to the presence of diffusion in (\ref{odepde1a}) and (\ref{odepde1b}). We note that model shows that in the long term, if there is no external source of the infection, over the long term it will disappear in the host, and the host population will survive the disease.
\end{remark}

\section{Spatially Inhomogeneous Reservoir Population}
In this section we analyze the qualitative behavior of a diffusive SI model that includes demographics. This model will subsequently be used as component of a spatial-temporal host reservoir model.  We begin with a discussion of generic spatially inhomogeneous diffusive logistic partial differential equations of the form,
\begin{subequations}
\begin{align}
\frac{\partial v}{\partial t}&=\bigtriangledown \cdot d(x)\bigtriangledown v+g(x,t)v-r(x)v^2&x\in\Omega,t>0\label{4-1a}\\
\frac{\partial v}{\partial \eta}&=0&x\in\partial\Omega,t>0\label{4-1b}\\
v(x,0)&=v_0(x)&x\in\Omega\label{4-1c}
\end{align}
\end{subequations}
These equations are the commonly called Fisher Kolmogorov equations in the case when $g(x,t)$ is independent of $t$. The FK equations are well known in the literature and frequently arising in population dynamics, ecology, and population genetics as well as in other contexts. Our interest lies in the special case when there exists $a\in C(\overline\Omega)$ such that $\|g(\cdot,t)-a(\cdot)\|_{\infty,\Omega}\to 0$ as $t\to\infty$. In particular, we are interested in the relation between the solution to (\ref{4-1a}-\ref{4-1c}) and solutions to 
\begin{subequations}
\begin{align}
-\bigtriangledown \cdot d(x)\bigtriangledown u&=a(x)u-r(x)u^2&x\in\Omega\label{4-2a}\\
\frac{\partial u}{\partial \eta}&=0&x\in\Omega\label{4-2b}
\end{align}
\end{subequations}
It turns out that the question is intimately tied to the eigenvalue problem
\begin{subequations}
\begin{align}
\bigtriangledown\cdot d(x)\bigtriangledown \xi+a(x)\xi&=\Lambda\xi&x\in\Omega\label{4-3a}\\
\frac{\partial\xi}{\partial\eta}&=0&x\in\partial\Omega\label{4-3b}
\end{align}
\end{subequations}
\begin{theorem}\label{theorem5} 
Assume $\Omega$ is a bounded region in $R^2$ with smooth boundary $\partial\Omega$ such that $\Omega$ lies locally on one side of $\partial\Omega$, $g\in C(\overline\Omega\times R)$ and $a\in C(\overline\Omega)$ such that $\|g(\cdot,t)-a(\cdot)\|_{\infty,\Omega}\to 0$ as $t\to\infty$, and $d\in C^2(\overline\Omega)$ and $r\in C(\overline\Omega)$ with $0<\alpha\min_{x\in\overline\Omega}\left\{d(x),r(x)\right\}$. If $v_0\in C(\overline\Omega)$ with $v_0\ge$ and not identically $0$, then there exists a unique, nonnegative classical solution to (\ref{4-1a}-\ref{4-1c}) on $\Omega\times(0,\infty)$ with
$$0<v(x,t)\le\max\left\{\|v_0\|_{\infty,\Omega},\frac{\|a\|_{\infty,\Omega}}{r_{min}}\right\}\text{ for all }x\in\Omega,t>0$$ 
In addition, if the principal eigenvalue of (\ref{4-3a}-\ref{4-3b}) is positive, then there is a unique positive solution $u$ to (\ref{4-2a}-\ref{4-2b}), and $\|v(\cdot,t)-u(\cdot)\|_{\infty,\Omega}\to 0$ as $t\to\infty$. If the principal eigenvalue of (\ref{4-3a}-\ref{4-3b}) is negative, then there is no positive solution to (\ref{4-2a}-\ref{4-2b}), and $\|v(\cdot,t)\|_{\infty,\Omega}\to 0$ exponentially as $t\to\infty$.
\end{theorem}
\begin{proof}
The results associated with (\ref{4-2a}-\ref{4-2b}) can be found in Chapter 3 of \cite{bib16}, and also \cite{bib17}. For the remaining portion, let $h(\epsilon,x,y)=a(x)+\epsilon-r(x)y$ for $(\epsilon,x,y)\in R\times\Omega\times R$, and suppose $U^\epsilon$ solves
\begin{subequations}
\begin{align}
\frac{\partial U^\epsilon}{\partial t}&=\bigtriangledown \cdot d(x)\bigtriangledown U^\epsilon+h(\epsilon,x,U^\epsilon)U^\epsilon&x\in\Omega,t>0\label{4-4a}\\
\frac{\partial U^\epsilon}{\partial \eta}&=0&x\in\partial\Omega,t>0\label{4-4b}\\
U^\epsilon(x,0)&=v_0(x)&x\in\Omega\label{4-4c}
\end{align}
\end{subequations}
Suppose the principal eigenvalue of (\ref{4-3a}-\ref{4-3b}) is positive. Then there exists $\alpha>0$ so that for each $\epsilon\in R$ with $|\epsilon|>\alpha$, the principal eigenvalue for (\ref{4-3a}-\ref{4-3b}) with $a(x)$ replaced by $a(x)+\epsilon$ is positive, and consequently, there exists a unique positive soluition $u^\epsilon$ to
\begin{subequations}
\begin{align}
-\bigtriangledown \cdot d(x)\bigtriangledown u^\epsilon&=h(\epsilon,x,u^\epsilon)u^\epsilon&x\in\Omega,t>0\label{4-5a}\\
\frac{\partial u^\epsilon}{\partial \eta}&=0&x\in\partial\Omega,t>0\label{4-5b}
\end{align}
\end{subequations}
Also, from results in \cite{bib16}, $\|U^\epsilon(\cdot,t)-u^\epsilon(\cdot)\|_{\infty,\Omega}\to 0$ as $t\to\infty$, and since $g(\epsilon,x,y)$ is strictly increasing in $\epsilon$, results in \cite{bib16} imply $u^\epsilon(x)$ is increasing in $\epsilon$., and consequently, $U^{-\epsilon}(\cdot,t)\le v(\cdot,t)\le U^\epsilon(\cdot,t)$ for $0<\epsilon<\alpha$ and $t$ sufficiently large.

Now consider the eigenvalue problem
\begin{subequations}
\begin{align}
\bigtriangledown\cdot d(x)\bigtriangledown \psi+h(0,x,u)\psi&=\Lambda\psi&x\in\Omega\label{4-6a}\\
\frac{\partial\psi}{\partial\eta}&=0&x\in\partial\Omega\label{4-6b}
\end{align}
\end{subequations}
Note that $\Lambda=0$ and $\psi=u$ solve this problem. In addition, $\psi=u>0$. Therefore, $\Lambda=0$ is the principal eigenvalue for (\ref{4-6a}-\ref{4-6b}). Consequently, the principal eigenvalue of
\begin{subequations}
\begin{align}
\bigtriangledown\cdot d(x)\bigtriangledown \psi+\left[h(0,x,u)+u\frac{\partial}{\partial y}h(0,x,u)\right]\psi&=\Lambda\psi(x)&x\in\Omega\label{4-7a}\\
\frac{\partial\psi}{\partial\eta}&=0&x\in\partial\Omega\label{4-7b}
\end{align}
\end{subequations}
is negative since $u(x)\frac{\partial}{\partial y}h(0,x,u)=-u(x)r(x)<0$. Therefore, the implicit function theorem implies that for $|\epsilon|<\alpha$, we have $\|u^\epsilon-u\|_{\infty,\Omega}\to 0$ as $\epsilon\to 0$, which implies $\|v(\cdot,t)-u(\cdot)\|_{\infty,\Omega}\to 0$ as $t\to\infty$.

The case when the principal eigenvalue for (\ref{4-3a}-\ref{4-3b}) is negative can be handled by combining the comparison principle and results in Chapter 3 of \cite{bib16}. 
\end{proof}
\begin{remark}\label{remark1}
If $\int_\Omega a(x)dx>0$ then the principal eigenvalue for (\ref{4-3a}-\ref{4-3b}) is positive, and this is certainly the case when a(x) is nonnegative and positive on a subset of $\Omega$ of positive measure.
\end{remark}
The pathogens causing the disease are harbored in the reservoir host. The reservoir provides an environment in which they naturally live and reproduce. We will be modeling the spread of a non-lethal virus through a population that remains confined to a geographic region which mathematically described as a bounded region of $R^2$. The pathogens are assumed to be transmitted horizontally through the reservoir population with no vertical transmission to offspring. The presence of pathogens has no negative impact on the reservoir.  The population divides into two compartments, the susceptible and the infected.  The time dependent population densities of the these two compartments are represented by $\phi(x,t)$ and $\psi(x,t)$, respectively. The total population density is given by $\theta(x,t)=\phi(x,t)+\psi(x,t)$. The dispersion of the population in each compartment is modeled by Fickian diffusion with diffusivity $d(x)$. The infection is transmitted through the vector reservoir by a spatially heterogeneous force of infection,  $f(x,\phi,\psi)=\sigma(x)\phi\psi$. Individuals that become infected remain infected with no recovery. Each of the compartments are subject to depletion by spatially dependent logistic mortality $m(x)$ with the uninfected population replenished by a spatially dependent birth term which depends linearly on the total population density.  Such models can be called spatially dependent logistic Susceptible/Infective (SI) models, and they give rise to the following coupled system of reaction diffusion type equations.
\begin{subequations}
\begin{align}
\frac{\partial \theta}{\partial t}&=\bigtriangledown\cdot d_1(x)\bigtriangledown\theta+\beta(x)\theta-m(x)\theta^2&x\in\Omega,t>0\label{4-8a}\\
\frac{\partial \phi}{\partial t}&=\bigtriangledown\cdot d_1(x)\bigtriangledown\phi+\beta(x)\theta-\sigma_1(x)\phi\psi-m(x)\theta\phi&x\in\Omega,t>0\label{4-8b}\\
\frac{\partial \psi}{\partial t}&=\bigtriangledown\cdot d_1(x)\bigtriangledown\psi+\sigma_1(x)\phi\psi-m(x)\theta\phi&x\in\Omega,t>0\label{4-8c}\\
\frac{\partial \theta}{\partial\eta}&=\frac{\partial \phi}{\partial\eta}=\frac{\partial \psi}{\partial\eta}=0&x\in\partial\Omega,t>0\label{4-8d}\\
\theta(x,0)&=\theta_0(x);\phi(x,0)=\phi_0(x);\psi(x,0)=\psi_0(x)&x\in\Omega\label{4-8e}
\end{align}
\end{subequations}
We require that the following be satisfied.

\medskip
\noindent\textbf{Condition II:}
\begin{itemize}
\item $\Omega$ is a bounded region in $R^2$ with smooth boundary $\partial\Omega$ such that $\Omega$ lies locally on one side of $\partial\Omega$.
\item $d_1\in C^2(\overline\Omega)$ and $d_1>0$ on $\overline\Omega$.
\item $\sigma_1,\beta,m\in C^1(\overline\Omega)$ and $\sigma_1,\beta,m>0$ on $\overline\Omega$. 
\item $\phi_0,\psi_0\in C(\overline\Omega)$, $\phi_0,\psi_0>0$ on $\overline\Omega$, and $\theta_0=\phi_0+\psi_0$.
\end{itemize}

We have the following result.

\begin{theorem}\label{theorem6}
Assume Condition II is satisfied. Then there exists a unique uniformly bounded solution triple $(\theta,\phi,\psi)$ to (\ref{4-8a}-\ref{4-8e}), $\theta=\phi+\psi$ and 
$$0\le\theta,\phi,\psi\le\max\left\{\|\theta_0\|_{\infty,\Omega},\frac{\|\beta\|_{\infty,\Omega}}{m_{\min}}\right\}$$
In addition, $\lim_{t\to\infty}\|\theta(\cdot,t)-\theta_*(\cdot)\|_{\infty,\Omega}=0$, where $\theta_*$ is the unique positive solution to
\begin{align}
-\bigtriangledown\cdot d_1(x)\bigtriangledown\theta_*&=\beta(x)\theta_*-m(x)\theta_*^2&x\in\Omega,t>0\label{theta-eq}\\
\frac{\partial \theta_*}{\partial\eta}&=0&x\in\partial\Omega\nonumber
\end{align}
\end{theorem}

\begin{proof}
One can apply variety of different arguments to establish the existence of a maximal interval of existence $[0,T_{\max})$ and guarantee on this interval $\theta,\phi,\psi\ge 0$. This reduced the problem to one of establishing uniform a priori bounds for $\theta,\phi,\psi$ on $[0,T_{\max})$. Theorem 1 of \cite{bib15} insures that $0\le\theta\le\max\left\{\|\theta_0\|_{\infty,\Omega},\frac{\|\beta\|_{\infty,\Omega}}{m_{\min}}\right\}$. We add equations (\ref{4-8b}-\ref{4-8c}), observing that $\theta(x,t)=\phi(x,t)+\psi(x,t)$ and using the nonnegativity of solutions components to complete the argument that solutions are uniformly bounded and globally well posed. Theorem \ref{theorem5} and Remark \ref{remark1} imply $\lim_{t\to\infty}\|\theta(\cdot,t)-\theta_*(\cdot)\|_{\infty,\Omega}=0$.
\end{proof}

We consider $\theta_*$, the positive steady state solution of (\ref{theta-eq}), as a predetermined quantity, and introduce the function $R(x)=\theta_*(x)(\sigma_1(x)-m(x))$.  We shall see that if $a(x)=R(x)$ in (\ref{4-2a}-\ref{4-2b}), and the associated principal eigenvalue of (\ref{4-3a}-\ref{4-3b}) is negative, then the reservoir converges to a pathogen free steady state.  However, if the associated principal eigenvalue is positive, then the pathogen persists in the reservoir and we are assured that $\phi(\cdot,t)$,and $\psi(\cdot,t)$  converge to a positive endemic steady state. In the next section we shall also see that this persistence or non-persistence determines whether or not the susceptible host population survives the infection.  

The uniform a priori $L_\infty$ bounds on $\theta$, $\phi$ and $\psi$ insure that their time derivatives, gradients, and the time derivative of their gradients all have uniform a prioiri $L_\infty$ bounds, cf \cite{bib14}, p. 226.

\begin{proposition}\label{proposition4-2}
Suppose the conditions of Theorem \ref{theorem6} are satisfied, $\theta,\phi,\psi$ solve (\ref{4-8a}-\ref{4-8e}), and $\tau>0$. Then the sup norms over $\Omega$ of  $\frac{\partial \theta}{\partial t}$, $\frac{\partial \phi}{\partial t}$, $\frac{\partial \psi}{\partial t}$, $|\bigtriangledown\theta|$, $|\bigtriangledown\phi|$, $|\bigtriangledown\psi|$, $\frac{\partial}{\partial t}|\bigtriangledown\theta|$, $\frac{\partial}{\partial t}|\bigtriangledown\phi|$ and $\frac{\partial}{\partial t}|\bigtriangledown\psi|$ are bounded independent of $t$, for all $t\ge\tau$.
\end{proposition}

We are now in a position to characterize the long term behavior of $\theta$, $\phi$ and $\psi$.

\begin{theorem}\label{theorem7}
Assume the conditions of Theorem \ref{theorem6} hold, $R(x)=\theta_*(x)(\sigma_1(x)-m(x))$, and let $\Lambda_0$ be the principal eigenvalue of (\ref{4-3a}-\ref{4-3b}) associated iwth $a(x)=r(x)$. If $\Lambda_0>0$, then there exists a unique positive endemic steady state $\phi_*,\psi_*$ such that
\begin{align}
-\bigtriangledown\cdot d_1(x)\bigtriangledown\phi_*&=\beta(x)\theta_*-\sigma_1(x)\phi_*\psi_*-m(x)\theta_*\phi_*&x\in\Omega\nonumber\\
-\bigtriangledown\cdot d_1(x)\bigtriangledown\psi_*&=\sigma_1(x)\phi_*\psi_*-m(x)\theta_*\phi_*&x\in\Omega\nonumber\\
\frac{\partial \phi_*}{\partial\eta}&=\frac{\partial \psi_*}{\partial\eta}=0&x\in\partial\Omega\nonumber
\end{align}
If $\Lambda_0<0$, then there is no nontrivial endemic steady state, and we have
$$\lim_{t\to\infty}\|\phi(\cdot,t)-\theta_*(\cdot)\|_{\infty,\Omega}=\lim_{t\to\infty}\|\psi(\cdot,t)\|_{\infty,\Omega}=0$$
\end{theorem}

\begin{proof}
Theorem \ref{theorem5} insures a unique strictly positive steady state solution $\theta_*$ to (\ref{4-8a}-\ref{4-8b}) such that $\lim_{t\to\infty}\|\theta(\cdot,t)-\theta_*(\cdot)\|_{\infty,\Omega}=0$. Using the fact that $\theta=\phi+\psi$, we reduce the system to the two component system
\begin{subequations}
\begin{align}
\frac{\partial\theta}{\partial t}&=\bigtriangledown\cdot d_1(x)\bigtriangledown\theta+\beta(x)\theta-m(x)\theta^2&x\in\Omega,t>0\label{4-9a}\\
\frac{\partial\psi}{\partial t}&=\bigtriangledown\cdot d_1(x)\bigtriangledown\psi+\theta(\sigma_1(x)-m(x))\psi-\sigma_1(x)\psi^2&x\in\Omega,t>0\label{4-9b}
\end{align}
\end{subequations}
Theorem \ref{theorem5} implies solutions to (\ref{4-9a}) converge uniformly to $\theta_*(x)>0$, so we confine our analysis to (\ref{4-9b}), subject to (\ref{4-8a}-\ref{4-8e}). Note that $\|\theta(\cdot,t)(\sigma_1(\cdot)-m(\cdot))-R(\cdot)\|_{\infty,\Omega}\to 0$ as $t\to\infty$. Consequently, we can use Theorem \ref{theorem5} to obtain the result.
\end{proof}

We can say more concerning the convergence of $\psi(\cdot,t)$ as $t\to\infty$. The following result will be useful in the next section.

\begin{corollary}\label{corollary4-3}
Assume the conditions of Theorem \ref{theorem7} hold. If $\Lambda_0<0$ then
$$0<\int_0^\infty\int_\Omega \psi(x,t)dxdt < \infty$$
\end{corollary}
\begin{proof}
The result follows immediately from Theorem \ref{theorem5} and the exponential convergence of $\psi(\cdot,t)$ to $0$ as $t\to\infty$.
\end{proof}
\section{Spatial Spread of a Vector Reservoir Supported Virus}
In this section we consider the impact of a localized infected reservoir on the spread of the virus among a disturbed host population across a larger population.  The host population inhabits and is confined to a bounded region $\Omega$ in $R^2$ with smooth boundary $\partial\Omega$ such that $\Omega$ lies locally on one side of $\partial\Omega$. As in the preceding case, the variables $s$, $e$, $i$ and $c$ represent time dependent spatial densities of the susceptible, exposed, infective and contaminated classes. The time dependence populations are obtained by integration over $\Omega$.  The host dispersion across the region $\Omega$ is modeled by diffusion.  The diffusivities of susceptible and exposed hosts are given by  $d_2(x)$ and $d_3(x)$ respectively. The reservoir population remains confined to a proper sub region $\Omega_*$  of $\Omega$ with smooth boundary $\partial\Omega_*$ such that $\Omega_*$ lies locally on one side of $\partial\Omega_*$. Homogeneous Neumann Boundary Conditions are imposed on the hosts and the vector reservoir on the boundaries of  $\Omega$ and $\Omega_*$.  Issues of global well posedness and long term behavior of the reaction diffusion system modeling the dynamics of the disease in the reservoir population appeared in the previous section.  A spatially dependent force of infection accounts for transmission of the disease to the host from the infected reservoir.  This results in a force of infection of the form,
$$f_2(x,s,i,c,\psi)=\sigma_2(x)si+\kappa(x)s\tilde\psi+\omega(x)sc$$
Here $\sigma_2,\omega\in C^1(\overline\Omega)$ with $\sigma_2,\omega>0$, $\tilde\psi=\psi$ on $\Omega_*$ and $\tilde\psi=0$ on $\Omega- \overline\Omega_*$, and $\kappa\in L_\infty(\overline\Omega_*)$ is a $C^1$ function on $\overline\Omega_*$ with $\kappa>0$ on $\overline\Omega_*$ and $\kappa=0$ on $\Omega-\overline\Omega_*$. We have the following system that couples a reaction diffusion system on $\Omega_*$ with a partially dissipative system on $\Omega$.
\begin{subequations}
\begin{align}
\frac{\partial\theta}{\partial t}&=\bigtriangledown\cdot d_1(x)\bigtriangledown\theta+\beta(x)\theta-m(x)\theta^2&x\in\Omega_*,t>0\label{5-1a}\\
\frac{\partial\phi}{\partial t}&=\bigtriangledown\cdot d_1(x)\bigtriangledown\phi+\beta(x)\theta-\sigma_1(x)\phi\psi-m(x)\theta\phi&x\in\Omega_*,t>0\label{5-1b}\\
\frac{\partial\psi}{\partial t}&=\bigtriangledown\cdot d_1(x)\bigtriangledown\psi+\sigma_1(x)\phi\psi-m(x)\theta\psi&x\in\Omega_*,t>0\label{5-1c}\\
\frac{\partial\theta}{\partial \eta}&=\frac{\partial\phi}{\partial \eta}=\frac{\partial\psi}{\partial \eta}=0&x\in\partial\Omega_*,t>0\label{5-1d}\\
\theta(x,0)&=\theta_0(x);\psi(x,0)=\phi_0(x);\psi(x,0)=\psi_0(x)&x\in\Omega_*\label{5-1e}\\
\frac{\partial s}{\partial t}&=\bigtriangledown\cdot d_2(x)\bigtriangledown s-(\sigma(x)si+\omega(x)sc+\kappa(x)s\tilde\psi)&x\in\Omega,t>0\label{5-1f}\\
\frac{\partial e}{\partial t}&=\bigtriangledown\cdot d_3(x)\bigtriangledown e+(\sigma(x)si+\omega(x)sc+\kappa(x)s\tilde\psi)-\lambda e&x\in\Omega,t>0\label{5-1g}\\
\frac{\partial i}{\partial t}&=\lambda_1 e-\gamma i&x\in\Omega,t>0\label{5-1h}\\
\frac{\partial c}{\partial t}&=\gamma_1 i-\mu c i&x\in\Omega,t>0\label{5-1j}\\
\frac{\partial s}{\partial \eta}&=\frac{\partial e}{\partial \eta}=0&x\in\partial\Omega,t>0\label{5-1k}\\
s(x,0)&=s_0(x);e(x,0)=e_0(x);i(x,0)=i_0(x);c(x,0)=c_0(x)&x\in\Omega\label{5-1l}
\end{align}
\end{subequations} 
We introduce the following requirements for the system that couples the spatial dynamics of the reservoir on subdomain $\Omega_*$ with the dynamics of the host on the larger domain $\Omega$.

\bigskip
\noindent\textbf{Condition III:}
\begin{itemize}
\item $\Omega$ is a bounded region in $R^2$ with smooth boundary $\partial\Omega$ such that $\Omega$ lies locally on one side of $\partial\Omega$, and $\Omega_*$ is a proper subregion of $\Omega$ with smooth boundary $\partial\Omega_*$ such that $\Omega_*$ lies locally on one side of $\partial\Omega_*$. 
\item $d_1\in C^2(\overline\Omega_*)$ and $d_1>0$.
\item $\sigma_1,\beta,m\in C^1(\overline\Omega_*)$ and $\sigma_1,\beta,m>0$. 
\item $\kappa\in\L_\infty(\Omega)$ and $\kappa\in C^1(\overline\Omega_*)$ such that $\kappa>0$ on $\overline\Omega_*$ and $\kappa=0$ on $\Omega-\overline\Omega_*$.
\item $d_2,d_3 \in C^2(\overline\Omega)$ with $d_2,d_3>0$.
\item $\lambda,\lambda_1,\lambda_2,\gamma,\gamma_1,\gamma_2,\mu>0$ with $\lambda=\lambda_1+\lambda_2$ and $\gamma=\gamma_1+\gamma_2$.
\item $\omega,\sigma\in C^1(\overline\Omega)$ with $\omega,\sigma>0$.
\item $\theta_0,\phi_0,\psi_0\in C^2(\overline\Omega_*,(0,\infty))$ with $\theta_0=\phi_0+\psi_0$.
\item $s_0,e_0,i_0,c_0\in C^2(\overline\Omega,(0,\infty))$.
\end{itemize}
We remark that if $\psi$ is not identically $0$, the function $\kappa(x)$ can create a spatial discontinuity of $f_2(x,s,i,c,\psi)$ across the boundary of $\partial\Omega_*$, so we are not able to obtain the classical smooth solutions that we have had with a smooth force of infection term. However, it is still possible to obtain so-called classical strong solutions.

\begin{defn}\label{definition5-2}
We say the functions $s,e,i,c$ on $\overline\Omega\times[0,\infty)$ and $\theta,\phi,\psi$ on $\overline\Omega_*\times[0,\infty)$ constitue a classical strong solution to (\ref{5-1a}-\ref{5-1l}) if and only if
\begin{itemize}
\item $s,e,i,c\in C(\overline\Omega\times[0,\infty))$ and $\theta,\phi,\psi\in C(\overline\Omega_*\times[0,\infty))$.
\item $\theta,\phi,\psi\in C^{2,1}(\Omega_*\times(\epsilon,T))\cap W_p^{2,1}(\Omega_*\times(\epsilon,T))$ for each $p>1$ and $0<\epsilon<T$.
\item $s,e\in W_p^{2,1}(\Omega\times(\epsilon,T))$ for each $p>1$ and $0<\epsilon<T$.
\item $i,c\in C^1([0,\infty),C(\overline\Omega))$.
\item The partial differential equations, boundary conditions, and the initial conditions (\ref{5-1a}-\ref{5-1l}) are satisfied.
\end{itemize}
\end{defn}
Our well-posedness result for the spatially heterogeneous host reservoir system follows.
\begin{theorem}\label{theorem8}
If the requirements of Condition III are satisfied, then there exists a unique classical strong componentwise nonnegative solution to (\ref{5-1a}-\ref{5-1l}). Furthermore, the solution is componentwise uniformly bounded in the sup norm.
\end{theorem}
\begin{proof}
We observe that there is no cross feedback from infected host to the uninfected reservoir.  So, the existence of a classical solution $\theta,\phi,\psi\in C^{2,1}(\Omega_*\times(0,T))\cap W_p^{2,1}(\Omega_*\times(\epsilon,T))$ for each $p>1$ and $0<\epsilon<T$ satisfying (\ref{5-1a}-\ref{5-1e}) is guaranteed by the analog of Theorem \ref{theorem6} on $\Omega_*$, and is independent of the evolution of the infection in the host. Furthermore, this solution is componentwise nonnegative and uniformly bounded in the sup norm. We can thereby assume $\theta,\phi,\psi$ as given. This allows us to focus on (\ref{5-1f}-\ref{5-1l}).  Standard arguments insure the local well posedness of nonnegative solutions on an interval $[0,T_{max})$, and $T_{max}=\infty$ if solutions are bounded in the sup norm on every finite time interval. The maximum principle immediately insures the existence of a $M_s>0$ such that $\|s(\cdot,t)\|_{\infty,\Omega}\le M_s$ for all $T>0$. Then the form of the right hand sides of the differential equations for $e$, $i$ and $c$ clearly result in $e$, $i$ and $c$ being bounded in the sup norm on every finite time interval, and we can conclude that $T_{max}=\infty$. We can now apply the arguments of developed in Theorem \ref{theorem3} to produce uniform $L_p(\Omega\times(k,k+1))$ bounds for each $p>1$ independent of the positive integer $k$ of the form
$$\|s\|_{p,\Omega\times(k,k+1)},\|e\|_{p,\Omega\times(k,k+1)},\|i\|_{p,\Omega\times(k,k+1)},\|c\|_{p,\Omega\times(k,k+1)}\le C_p$$
These in turn can be used to give uniform $W_{p,\Omega}^{2,1}(\Omega\times(k,k+1))$ bounds for each $p>1$ of the form 
$$\|s\|_{p,\Omega\times(k,k+1)}^{2,1},\|e\|_{p,\Omega\times(k,k+1)}^{2,1}\le K_p$$
Then, $p$ can be chosen sufficiently large to insure that $\|e\|_{\infty,\Omega\times(0,\infty)}$ is bounded. The uniform boundeds on $e$ allow us to obtain uniform bounds on $i$ and $c$ directly from their equations, and the result follows.
\end{proof}

The asymptotic impact of our vector reservoir on the host population hinges on the following eigenvalue problem.
\begin{subequations}
\begin{align}
\bigtriangledown\cdot d(x)\bigtriangledown \xi(x)+a(x)\xi(x)&=\Lambda\xi(x)&x\in\Omega_*\label{5-2a}\\
\frac{\partial\xi}{\partial\eta}&=0&x\in\partial\Omega_*\label{5-2b}
\end{align}
\end{subequations}
where $a(x)=\Omega_*(x)(\sigma_1(x)-m(x))$.
\begin{theorem}\label{theorem9}
If the conditions of Theorem \ref{theorem8} hold, then there exists a positive steady state $\theta_*$ of (\ref{5-1a}) and a constant $s_*\ge 0$ such that
$$\lim_{t\to\infty}\|s(\cdot,t)-s_*\|_{\infty,\Omega}=\lim_{t\to\infty}\|\theta(\cdot,t)-\theta_*(\cdot)\|_{\infty,\Omega_*}=0$$
In addition,
$$\lim_{t\to\infty}\|e(\cdot,t)\|_{\infty,\Omega}=\lim_{t\to\infty}\|i(\cdot,t)\|_{\infty,\Omega}=\lim_{t\to\infty}\|c(\cdot,t)\|_{\infty,\Omega}=0$$
Finally, if the principal eigenvalue of (\ref{5-2a}-\ref{5-2b}) is positive then there exists an endemic steady state $\phi_*,\psi_*>0$ on $\Omega_*$ so that
 $$\lim_{t\to\infty}\|\phi(\cdot,t)-\phi_*(\cdot)\|_{\infty,\Omega_*}=\lim_{t\to\infty}\|\psi(\cdot,t)-\psi_*(\cdot)\|_{\infty,\Omega_*}=0$$
and $s_*=0$.  If the principal eigenvalue of (5.2a-b) is negative, then $s_*>0$ and 
$$\lim_{t\to\infty}\|\phi(\cdot,t)-\theta_*(\cdot)\|_{\infty,\Omega_*}=\lim_{t\to\infty}\|\psi(\cdot,t)\|_{\infty,\Omega_*}=0$$
\end{theorem}
\begin{proof}
If the principal eigenvalue of (\ref{5-2a}-\ref{5-2b}) is positive, then an analog of Theorem \ref{theorem7} guarantees the uniform convergence of $\left(\theta(\cdot,t),\phi(\cdot,t),\psi(\cdot,t)\right)$ to a steady state solution $\left(\theta_*,\phi_*,\psi_*\right)$ of (\ref{5-1a}-\ref{5-1d}) such that $\theta_*,\phi_*,\psi_*>0$. Moreover, it insures that there exists a $\tilde k\in L_\infty(\overline\Omega_*)$ with $\tilde k$ positive and continuous on $\overline\Omega_*$ and $\tilde k=0$ on $\Omega-\overline\Omega_*$, and $T_*>0$ so that $\kappa(x)\psi(x,t)\ge\tilde k(x)$ for $t>T_*$. In this case, we have
$$\frac{\partial s}{\partial t}\le \bigtriangledown\cdot d_2(x)\bigtriangledown s-\tilde k(x)s,\text{  for  }x\in\Omega,t>T_*$$ 
We introduce the operator on $C(\overline\Omega)$ defined by 
$$(A_{\tilde k}u)(x)=(\bigtriangledown\cdot d_2 \bigtriangledown u)(x)-\tilde k(x)u(x)$$
$$D(A_{\tilde k})=\left\{u\in C(\overline\Omega) | u\in C^2(\Omega),\frac{\partial u}{\partial \eta}=0\text{ on }\partial\Omega\right\}$$
Then \cite{bib27,bib18,bib19} guarantee $A_{\tilde k}$ generates a positive analytic semigroup $\left\{T_{\tilde k}(t)|t\ge 0\right\}$ on $C(\overline\Omega)$ with $\|T_{\tilde k}(t)\|_{\infty,\Omega}\le \exp(-\epsilon t)$ for some $\epsilon>0$. Consequently, we have
$$\|s(\cdot,t)\|_{\infty,\Omega}\le\left(\|s_0\|_{\infty,\Omega}\right)\exp(-\epsilon t)$$
Hence, $\lim_{t\to\infty}\|s(\cdot,t)\|_{\infty,\Omega}=0$.

If we use the uniform bounds on $i$, $c$ and $\psi$, along with the exponential decay of $s(\cdot,t)$, then the positivity of $\lambda$ implies there are constants $K,\delta>0$ so that
$$\|e(\cdot,t)\|_{\infty,\Omega}\le K\exp(-\delta t)$$
Then thi sdecay estimate and equations (\ref{5-1h}-\ref{5-1j}) imply similar decay estimates for $i$ and $c$. That is,
$$\lim_{t\to\infty}\|e(\cdot,t)\|_{\infty,\Omega}=\lim_{t\to\infty}\|i(\cdot,t)\|_{\infty,\Omega}=\lim_{t\to\infty}\|c(\cdot,t)\|_{\infty,\Omega}=0$$
with the decay being exponential. 

If the principal eigenvalue of (\ref{5-2a}-\ref{5-2b}) is negative, then an analog of Theorem \ref{theorem7} guarantees uniform convergence of $\theta(\cdot,t)$ to a positive steady state solution $\theta_*$ of (\ref{5-1a}-\ref{5-1d}). In addition, $\lim_{t\to\infty}\|\phi(\cdot,t)-\theta_*\|_{\infty,\Omega_*}=0$ and $\lim_{t\to\infty}\|\psi(\cdot,t)\|_{\infty,\Omega_*}=0$, with the latter decay being exponential. As above, 
$$\lim_{t\to\infty}\|e(\cdot,t)\|_{\infty,\Omega}=\lim_{t\to\infty}\|i(\cdot,t)\|_{\infty,\Omega}=\lim_{t\to\infty}\|c(\cdot,t)\|_{\infty,\Omega}=0$$
We observe that $\overline s(t)$, the spatial average of $s(\cdot,t)$, is decreasing, and consequently, there exists $s_*\ge 0$ such that
$$\lim_{t\to\infty}\overline s(t)=\lim_{t\to\infty}\frac{1}{|\Omega|}\int_\Omega s(x,t)dx=s_*\ge 0$$
It remains to show that $s_*>0$. We define
$$f(x,t,s)=-\left(\sigma(x)i(x,t)+\omega(x)c(x,t)+\kappa(x)\psi(x,t)\right)s$$
By applying the sup norm bounds on $\sigma$, $\omega$ and $\kappa$, and the exponential decay of $\i(\cdot,t)$, $c(\cdot,t)$ and $\psi(\cdot,t)$, we can apply reasoning similar to the proof of Theorem \ref{theorem3} (applied only to the equation for $s$) to conclude that $s_*>0$. 
\end{proof}
\section{Conclusion and Further Discussion}
We have seen that in both the spatial independent case and the spatially dependent case, the infection will be extinguished over the long term in the host population if there is no interaction with the infected reservoir population. Indeed a high rate of mortality enhances the rapidity of the disappearance to the infection.  On the other hand contact with the infected reservoir can have dire consequences for the long term persistence of the susceptible host population.  We have observed in the spatially dependent case that the function $R(x)=\theta_*(x)(\sigma_1 (x)-m(x))$ and an eigenvalue problem play a critical role.  If the principal eigenvalue is negative, then the virus decays to zero in the reservoir. This is the case if $\sigma_1(x)-m(x)<0$ in $\Omega_*$. If it is positive it converges to a strictly positive endemic steady state. If the infection converges to a positive endemic steady state, we have seen that the susceptible host population converges to zero, and if the infection in the reservoir converges to zero the host persists and converges to a positive steady state. This is the case if $\sigma_1(x)-m(x)>0$ on $\Omega_*$. Here we are totally consistent with the spatially independent case where we set $R=\frac{\sigma_1}{m}$, and the number $R$ plays the role of the reproductive number.

There are issues with the applicability of this model to outbreak of the Ebola Virus in human populations.  Fruit bats are generally believed to be the natural hosts or reservoir of the Ebola virus with a wide range of mammals as well as humans serving as accidental hosts. The Ebola Virus has been implicated in the catastrophic decline in primate populations in western equatorial Africa.  We feel that our model may be most applicable in describing the transmission of Ebola from fruit bats to primates and other indigenous mammals.  It can be argued that a realistic model might include indirect environmental transmission as developed in \cite{bib13, bib20}.

In the case of the Ebola outbreaks in humans, the model should include a linkage to infected bush meat, \cite{bib22}.  In sub-Saharan Africa there is a wide spread tradition of harvesting, processing, and consuming bush meat \cite{bib1}.  The term bush meat covers a plethora of wild animals which includes non-human primates, rats and other rodents, antelopes, bush pigs, and pangolin as well as bats.  Subsequent work will address more complex models that incorporate a variety of intermediate hosts involved in the transfer of the Ebola Virus from fruit bats to humans as well as including indirect environmental transmission.

\pagebreak

\end{document}